\documentclass[12pt]{amsart}
\usepackage{amscd,amsmath,bm,latexsym,amssymb,amscd}
\usepackage[mathscr]{euscript} 
\usepackage[all]{xy}
\usepackage{layout}
\usepackage{pb-diagram} 
\usepackage{pstricks}
\usepackage{enumerate}
\usepackage{dsfont}
\usepackage{verbatim}  %%includes comment environment
\usepackage{color} %%used to get color text for corrections

\newtheorem{theorem}{Theorem}[section]
\newtheorem{corollary}[theorem]{Corollary}
\newtheorem{lemma}[theorem]{Lemma}
\newtheorem{proposition}[theorem]{Proposition}
\theoremstyle{definition}
\newtheorem{definition}[theorem]{Definition}
\theoremstyle{definition}
\newtheorem{example}[theorem]{Example}
\theoremstyle{definition}

\newtheorem{remark}[theorem]{Remark}

\newpsobject{grilla}{psgrid}{subgriddiv=1,griddots=10,gridlabels=6pt}


\def\Z{\mathbb{Z}}
\def\N{\mathbb{N}}
\def\C{\mathbb{C}}

\def\R{\mathbb{R}}

\def\I{\mathbb{I}}
\def\S{\mathbb{S}}
\def\P{\mathbb{P}}

\def\Co{\mathcal{C}}

\def\n{{\bf{n}}}
\def\m{{\bf{m}}}
\def\k{{\bf {k}}}
\def\0{{\bf{0}}}
\def\1{{\bf{1}}}
\def\hofib{\displaystyle\mathop{\textup{hofib}}}
\DeclareMathOperator{\Hom}{\textup{Hom}}
\DeclareMathOperator{\ucolim}{colim}
\DeclareMathOperator{\uhocolim}{hocolim}

\def\Top{\textup{Top}}

\newcommand{\john}[1]{{#1}}

\newcommand{\ul}[1]{{#1}}

\begin{document}

\title[Infinite loop spaces and nilpotent K-theory] 
{Infinite loop spaces and nilpotent K-theory}

\author[A.~Adem]{Alejandro Adem$^{*}$}
\address{Department of Mathematics,
University of British Columbia, Vancouver BC V6T 1Z2, Canada}
\email{adem@math.ubc.ca}

\author[J.~M.~G\'omez]{Jos\'e Manuel G\'omez$^{*}$}
\address{Departmento de Matem\'aticas,
Universidad Nacional de Colombia, Medell\'in, AA 3840, Colombia}
\email{jmgomez0@unal.edu.co}
\thanks{$^{*}$The first author was supported by NSERC. The second author acknowledges the financial 
support of COLCIENCIAS through grant number 121565840569 of
the Fondo Nacional de Financiamiento para la Ciencia, la
Tecnolog\'ia  y la Inovaci\'on, Fondo Francisco Jos\'e de Caldas.}

\author[J.~A.~Lind]{John A. Lind}
\address{Department of Mathematics,
Johns Hopkins University, Baltimore, MD 21218, USA \\
Current address: Universit\"at Regensburg, 93040 Regensburg, Germany}
\email{john.alexander.lind@gmail.com}

\author[U.~Tillmann]{Ulrike Tillmann}
\address{Mathematical Institute, 
Oxford University,  Oxford OX2 6GG, UK}
\email{tillmann@maths.ox.ac.uk}

\begin{abstract}

Using a construction 
derived from the descending central series of the free groups, 
we produce filtrations by infinite loop spaces of the classical 
infinite loop spaces  $BSU$, $BU$, $BSO$, 
$BO$, 
$BSp$,  $BGL_{\infty}(R)^{+}$ 
and $Q_0(\S^{0})$. We show that these infinite loop spaces are 
the zero spaces of \john{non-unital} $E_\infty$-ring spectra.
We introduce the notion of $q$-nilpotent K-theory 
of a CW-complex $X$ for any $q\ge 2$, 
which extends the notion of commutative 
K-theory defined by Adem-G\'omez, 
and show that it
is represented by $\mathbb Z\times B(q,U)$, were $B(q,U)$ is the
$q$-th term of the aforementioned filtration of $BU$. 
 
For the proof  we introduce an alternative way of associating 
an infinite loop space
to a commutative $\I$-monoid and give criteria when it can be identified
with the plus construction on the associated limit space.
Furthermore, we introduce the notion of a commutative $\I$-rig and 
show that they give rise to \ul{non-unital}
$E_\infty$-ring spectra. 
\end{abstract}

\maketitle
\tableofcontents

\section{Introduction}

Let $G$ denote a locally compact, Hausdorff topological group such
that $1_{G}\in G$ is a non-degenerate 
base point. It is well-known that we can obtain a model for the 
classifying space $BG$ 
as the geometric realization of the classical bar construction 
$B_{*}G$. Now fix an integer $q\ge 2$ and let $\Gamma^{q}_{n}$ be the 
$q$-th stage of the descending central series of  the free group on 
$n$-letters $F_{n}$, with the convention $\Gamma_{n}^{1}=F_{n}$. 
Consider the set of homomorphisms $B_{n}(q,G):=\Hom(F_{n}/\Gamma_{n}^{q},G)$. 
If $e_{1},\dots, e_{n}$ are generators of $F_{n}$, then evaluation  
on the classes corresponding to $e_{1},\dots, e_{n}$ provides 
a natural inclusion $B_{n}(q,G)\subset G^{n}$. 
Using this inclusion we  can give 
$B_{n}(q,G)$ the subspace topology.  Therefore $B_{n}(q,G)$ 
is precisely the space of ordered $n$-tuples in $G$ generating a subgroup 
of $G$ with nilpotence class less than $q$. For each fixed $q\ge 2$  
the collection $\{B_{n}(q,G)\}_{n\ge 0}$ 
forms a simplicial space with face and degeneracy maps induced by those in the  
bar construction.
The geometric realization of this simplicial space is denoted by  $B(q,G)$.  
These spaces were first introduced in \cite{ACT}, where many of
their basic properties were established. 
They give rise to a natural filtration of the classifying space
\[
B(2,G)\subset B(3,G)\subset \cdots\subset B(q,G)\subset B(q+1,G)\subset \cdots 
\subset BG.
\]
For $q=2$  we obtain $B_{com}G:=B(2,G)$ which is constructed by 
assembling  the different spaces of ordered commuting $n$-tuples in the 
group $G$.                                                                               
In \cite{AG} it was shown that for Lie groups this 
space plays the role of a classifying space for commutativity. 
More generally  $B(q,G)$ is a classifying space for $G$-bundles of transitional nilpotency class
less than $q$.

For the infinite unitary group $U=\ucolim_{n\to \infty}U(n)$, it is 
well known that $BU$ is the infinite loop space underlying a \john{non-unital $E_\infty$-ring spectrum, namely the homotopy fiber of the Postnikov section $ku \longrightarrow H\Z$.  In other words, 
$BU$ is a so called non-unital $E_\infty$-ring space. }
A basic question is whether the above  gives rise to a filtration of $BU$ by non-unital
$E_\infty$-ring spaces.
The main purpose of this paper is to show that indeed this is the case, not only for 
$U$ but also for other linear groups.

\begin{theorem}
The spaces $B(q,SU)$, $B(q,U)$, $B(q,SO)$ 
$B(q,O)$ 
and  $B(q, Sp)$ provide a 
filtration by \john{non-unital} $E_\infty$-ring spaces of the classical infinite loop spaces 
$BSU$, $BU$, $BSO$,
$BO$ 
and $BSp$ respectively. 
\end{theorem}

The  $q$-nilpotent 
K-theory of a space $X$ is 
defined using
isomorphism classes of bundles on $X$
whose transition functions generate subgroups of
nilpotence class less than $q$. 
We show
that  
$K_{q-nil}(X)\cong [X,\mathbb Z \times B(q,U)]$, from which we obtain

\begin{corollary}
$K_{q-nil}(-)$ is the 
zero-th term of a
generalized multiplicative cohomology theory.
\end{corollary}
\noindent In particular we obtain a sequence of multiplicative cohomology theories
\[
K_{com}(X)=K_{2-nil}(X)\to K_{3-nil}(X)\to \cdots \to K_{q-nil}(X)\to \cdots \to K(X).
\]
We also show that $B(q,U)\to BU$ splits as a map of infinite loop spaces,
whence we see that topological K-theory is a direct summand in $K_{q-nil}$.

The infinite loop space structure on $B(q,G)$ for $G=U, SU, SO, O, Sp$ is obtained by using 
the machinery of 
commutative $\I$-monoids first introduced by B\"okstedt and developed in 
\cite {CS}, \cite {SS}, \cite {Lind}. 
Here $\I$ is the category of finite sets and injections. 
In addition to the usual construction, we associate an infinite loop space to a commutative 
$\I$-monoid by restricting the usual homotopy colimit construction to 
the  subcategory $\P$ of finite sets and isomorphisms.
This allows us to identify the  homotopy type of the homotopy colimit under certain conditions. 
Another addition to infinite loop space theory is
the introduction of the notion of a commutative $\I$-rig which we show to give rise 
to a bipermutative category and hence an $E_\infty$-ring spectrum.

Our main examples above all arise from commutative $\I$-rigs 
where we can identify the infinite loop space
as the plus construction of the associated limit space.
A more complicated situation 
arises for $Q_0(\mathbb S^0)\simeq B\Sigma_{\infty}^+$ and 
$BGL_{\infty}(R)^{+}$. Our methods give rise to 
natural sequences of $E_\infty$-ring spaces but the terms are not easy to describe.

The outline of this article is as follows. In Section \ref{commutative I-monoids} 
we use the machinery of \john{
commutative $\I$-monoids to produce two associated  infinite loop 
spaces, one of which is a non-unital $E_{\infty}$ ring space when the $\I$-monoid is an $\I$-rig}. In Section \ref{identifying} 
we show that these are homotopy equivalent  and identify them
under suitable assumptions. 
Then in Section \ref{applications} 
we apply these results to prove Theorem 1.1 and show that the spaces $B(q,U)$ for $q\ge 2$ are
infinite loop spaces and that $BU$ splits off. Finally, in Section \ref{K theory} 
we introduce the notion of $q$-nilpotent K-theory and show that it is
represented by the infinite loop spaces $\mathbb Z\times B(q,U)$,
answering the question raised for commutative K-theory in \cite{AG}.

The authors would like to thank Christian Schlichtkrull for helpful 
conversations about commutative $\I$-monoids, Simon Gritschacher for
drawing our attention to \cite {FiedoroOgle} and the referee for providing very useful comments.
   
\section{Commutative $\I$-monoids  and infinite loop spaces}\label{commutative I-monoids}

The standard construction of the infinite loop space structure  on $BU$
from the permutative category of complex vector spaces and their isomorphisms 
does not restrict to give an infinite loop space structure on $B(q, U)$. 
Instead we are going to use certain constructions on commutative $\I$-monoids.
More precisely, we will give  two constructions of
permutative categories from commutative $\I$-monoids.
For the case of interest the permutative categories are actually  bipermutative and
hence give rise to $E_\infty$-ring spectra.
We start  by setting up some notations and basic definitions 
following  \cite{CS},  \cite{SS}, and  \cite{Lind}. We will use \cite{EM} as a 
reference for bipermutative categories and the associated  
multiplicative infinite loop space machinery. 

\subsection{The category $\I$ and its subcategories $\P$ and $\N$}
These three categories are skeletons of the category of finite sets and injections, 
the category of finite sets and isomorphisms, and 
the translation category associated to the monoid
of natural numbers. We will use the following notation.
 
For every integer $n\ge 0$ let $\n$ denote 
the set $\{1,2,\dots,n\}$. When $n=0$  we use the  
convention ${\bf{0}}:=\emptyset$.  Let $\I$ denote the category whose objects 
are the elements of the form $\n$ 
for all integers $n\ge 0$ with morphisms given 
by all injective maps. Note that in particular ${\bf{0}}$ is an initial object 
in the category  $\I$.   $\I$ is a
symmetric monoidal  category
under concatenation $\m\sqcup\n:=\{1,2,\dots, m+n\}$
with the symmetry
morphism given by the $(m,n)$-shuffle
map 
$$
\tau_{\m,\n}:\m\sqcup \n\to \n\sqcup \m.
$$ 
It 
is also symmetric monoidal
under Cartesian product 
$$
\m \times \n := \{ 1=(1,1), 2=(1,2), \dots, n+1= (2,1), \dots, mn=(m,n)\}
$$ 
given by lexicographic ordering. By definition
${\bf 0 }  \times \n = {\bf 0} = \n \times {\bf 0}$. The associated  
symmetry morphism is given by a permutation 
$$
\tau^\times_{\m \n}: \m \times \n  \to \n \times \m.
$$ 
The latter monoidal product is distributive over the former. More precisely
left distributivity  
$$
\delta ^l_{\m, \n, \k}:  
\m \times \k \sqcup \n \times \k
\to (\m \sqcup \n) \times \k 
$$ 
is given by the identity and right distributivity is  given by a permutation 
$$ 
\delta ^r_{\m, \n, \k}:
 \m \times \n \sqcup \m \times \k \to
 \m \times (\n \sqcup  \k).
$$
These two structures
make $\I$ into a bipermutative category in the sense of \cite[Def. 3.6] {EM}. 

$\I$ has two natural  subcategories. Let
$\P$ be the totally disconnected subcategory containing all objects and 
all isomorphisms $\sigma : \n \to \n$ but no other
morphisms, and let $\N$ denote the connected subcategory containing   
all objects, their identities and only the canonical
inclusions $j: \n \to \m$. While $\P$ is a bipermutative subcategory, 
$\N$ does not inherit any  monoidal structure from $\I$.

\subsection{Definitions of commutative  $\I$-monoids and $\I$-rigs}
An $\I$-space is a functor $X:\I\to \Top$. 
Every morphism in $\I$ 
can be factored as a composition of a canonical inclusion $j:\n\hookrightarrow \m$ 
and a permutation $\sigma:\m\to \m$. Therefore 
an $\I$-space $X:\I\to \Top$  determines a  sequence of spaces 
$X(\n)$  together with an induced  action of 
the symmetric group $\Sigma_{n}$ for $n\ge 0$, 
and structural 
maps $j_{n}:X(\n)\to X(\n+{\bf{1}})$ that are equivariant in 
the sense that $j_{n}(\sigma\cdot x)=\sigma \cdot j_{n}(x)$
for every $\sigma \in \Sigma_{n}$ and $x\in X(\n)$. On the right hand side 
we see $\sigma$ as element in $\Sigma_{n+1}$ via the canonical inclusion 
$\Sigma_{n}\hookrightarrow \Sigma_{n+1}$. 
Vice versa, given such a sequence of $\Sigma _n$-spaces $X(\n)$ and 
compatible structure maps $j_n$, they give rise to an $\I$-space
if and only if for $m \geq n$ and any two elements $ \sigma, \sigma ' \in \Sigma _m$ which 
restrict to the same permutation of $\n$ we have $\sigma (x) = \sigma '(x)$
for all $x \in j (X(\n))$.
We note that this condition is not satisfied by the sequence $X(\n) = \Sigma_n$ with the left 
or right  multiplication action, 
but {\it is} satisfied by the sequence 
$X(\n) = \n$ with the natural permutation action since  $\n \cong \I({\bf 1}, \n)$.

We say that an $\I$-space is an 
$\I$-monoid if it comes equipped with a natural transformation
\[
\mu_{\m,\n}:X(\m)\times X(\n)\to X(\m\sqcup \n)
\]
of functors defined on $\I \times \I$ and a natural transformation
\[
\eta_{\n} \colon \ast \to X(\n)
\]
from the constant $\I$-space $\ast(\n) = \ast$ 
to $X$ satisfying associativity and unit axioms for $* \in X(\0)$.  
We say that $X$ is a commutative $\I$-monoid 
if $\mu$ is commutative, meaning 
that the diagram
\begin{equation*}
\begin{CD}
X(\m)\times X(\n)@>\mu_{\m,\n}>> X(\m\sqcup \n)\\
@V\tau VV     @V\tau_{\m,\n} VV\\
X(\n)\times X(\m)@>\mu_{\n,\m}>>
X(\n \sqcup \m)
\end{CD}
\end{equation*}
commutes, where $\tau(x,y)=(y,x)$.  

An $\I$-rig is a commutative $\I$-monoid equipped with a natural transformation
\[
\pi_{\m, \n} : X(\m) \times X(\n)  \to X( \m \times \n)
\] 
\john{of functors defined on $\P \times \P$ and an element $1 \in X(1)$ satisfying associativity and unit axioms,} as well
as left distributivity, i.e. that the diagram 
$$
\begin{CD}
(X(\m)\times X(\n)) \times X(\k) @>\pi_{\m \sqcup \n, \k} 
\circ (\mu _{\m, \n} \times 1)>> X((\m\sqcup \n) \times \k )\\
@V (1\times \tau \times 1 )\circ ( 1 \times 1 \times \triangle) 
 VV     @A\delta^l_{\m,\n,\k} AA\\
X(\m)\times X(\k) \times X(\n) \times X(\k) @>\mu_{\m \times \k , \n \times \k} 
\circ (\pi_{\m, \k} \times \pi _{\n, \k}) >>
X(\m \times \k \sqcup \n \times \k) 
\end{CD}
$$
commutes, and right distributivity, which is given by an analogous 
commutative diagram. Here $\triangle$ is the diagonal map.  We emphasize that $\pi$ is only required to be natural on the subcategory  $\P \times \P$ of $\I \times \I$.\footnote{In fact, we do not know of any non-trivial examples where $\pi$ may be extended to a natural transformation of functors defined on $\I \times \I$.  The examples of $\I$-rigs that we discuss in \ul{Section}  \ref{subsec:examples} do not satisfy this additional naturality condition.  Indeed, as we will see in the following sections, an $\I$-rig that does satisfy this condition and has each level $X(\n)$ a connected space would give rise to a connected $E_{\infty}$-ring space $\uhocolim_{\I} X$.  An $E_{\infty}$-ring space whose multiplicative unit and additive unit lie in the same path component is contractible, so such examples would only give rise to trivial $E_{\infty}$-ring spectra.}

A commutative $\I$-rig is an $\I$-rig  
in which $\pi$ is commutative in the sense that
the diagram
$$ 
\begin{CD}
X(\m)\times X(\n)@>\pi_{\m,\n}>> X(\m\times \n)\\
@V\tau VV     @V\tau^\times_{\m, \n} VV\\
X(\n)\times X(\m)@>\pi_{\n,\m}>>
X(\n \times \m)
\end{CD}
$$
commutes.  A natural transformation $T$ between
two $\I$-spaces  $X$ and $Y$ defines a map of 
commutative $\I$-monoids ($\I$-rigs)
if it commutes with $\mu$ (and $\pi$) in the sense that $T\circ \mu_{\m , \n} = \mu_{\m,\n} 
\circ T\times T$ (and $T\circ \pi_{\m , \n} = \pi_{\m,\n} 
\circ T\times T$). 
We have thus defined a category of $\I$-spaces, a category of commutative $\I$-monoids, and a category of $\I$-rigs.

\subsection{Associated (bi)permutative translation categories}
We will use the following notation for translation categories.
If $Y \colon \Co \to \Top$ is a functor from a 
category $\Co$ to the category of topological spaces, 
we let $\Co \ltimes Y$ denote the translation category on $Y$.  
The translation category, also known as the Grothendieck construction, is a topological 
category whose objects are pairs $(c, x)$ consisting of an object $c$ of $\Co$ and a 
point $x \in Y(c)$.  A morphism in $\Co \ltimes Y$ from $(c, x)$ to $(c', x')$ is a 
morphism $\alpha \colon c \to c'$ in $\Co$ satisfying the equation $Y(\alpha)(x) = x'$.  
For example, if $\Co = G$ is a group, thought of as a one object category, then the 
translation category $G \ltimes Y$ is the action groupoid for the $G$-space $Y$ and 
its classifying space is the homotopy orbit space $B(G \ltimes Y) = EG \times_{G} Y$.  
In general, the classifying space $B(\Co \ltimes Y)$ is homeomorphic to the homotopy 
colimit $\uhocolim_{\Co} Y$ of $Y$ over $\Co$ defined using the bar construction.  
 
Suppose now that $X$ is a commutative $\I$-monoid.  Then the translation category 
$\I \ltimes X$ is a permutative category, as we now explain.  The monoidal 
structure $\oplus$ is defined on objects $(\m, x)$ and $(\n, y)$ by
\[
(\m, x) \oplus (\n, y) = (\m \sqcup \n, \mu_{\m, \n}(x, y)),
\]
and on morphisms $\alpha \colon (\m, x) \to (\m', x')$ and 
$\beta \colon (\n, y) \to (\n', y')$ by letting 
\[
\alpha \oplus \beta \colon (\m, x) \oplus (\n, y) \to (\m', x') \oplus (\n', y')
\]
be determined by the morphism
\[
\alpha \sqcup \beta \colon \m \sqcup \n \to \m' \sqcup \n'
\]
in the category $\I$.  Notice that 
$X(\alpha \sqcup \beta) (\mu_{\m, \n}(x, y)) = \mu_{\m', \n'}(x', y')$ 
by the naturality of $\mu$, so that this is well-defined.  The associativity and unit conditions 
for $X$ imply that $\I \ltimes X$ is a strict monoidal category with strict
unit object 
$(\mathbf{0}, \ast)$ determined by the unit $\eta$ of the $\I$-monoid $X$.  The 
commutativity of $X$ implies that $\I \ltimes X$ is a permutative category,
see for example \cite [Def 3.1]{EM}.  Note  that the  permutative structure on $\I \ltimes X$ 
restricts  to  the subcategory $\P \ltimes X$.

Suppose now that $X$ is a commutative $\I$-rig. 
Then by the same reasoning as above, there is another
permutative category structure on \john{$\P \ltimes X$} with product $\otimes$ 
induced by $\pi$ and strict unit object \john{$({\bf 1}, 1)$}. The distributivity axioms for 
$X$ translate to distributivity 
axioms for bipermutative categories, \cite [Def. 3.6] {EM}.  

Furthermore, a natural transformation $T$ between
two $\I$-spaces  $X$ and $Y$ induces a functor $\I \ltimes X \to \I\ltimes Y$.
If $X$ and $Y$ are commutative $\I$-monoids ($\I$-rigs) and $T$ is a morphism of such 
then the induced functor of 
 translation categories is a functor of 
(bi)permutative categories.

 We have thus proved the following result. 

\begin{proposition}
\john{The assignment $X \mapsto \I \ltimes X$ defines a functor from  the category 
of commutative $\I$-monoids to the category of permutative
categories, and the assignment $ X\mapsto \P \ltimes X$ defines a functor from the category of commutative $\I$-monoids ($\I$-rigs) to the category of (bi)permutative categories.}
\end{proposition}

\subsection{Construction of two infinite loop spaces}
Let $X$ be a commutative $\I$-monoid.
As explained in \cite{May1}, the classifying space of a permutative category is an 
$E_{\infty}$-space structured by an action of the Barratt-Eccles operad.  We have proved the next theorem.
\begin{theorem}\label{loop hocolim}
Suppose that $X:\I\to \Top$ is a commutative $\I$-monoid. Then the 
homotopy colimit 
$$
\uhocolim_{\I} X = B(\I \ltimes X)
$$ 
is an $E_{\infty}$-space.  %If $X$ is furthermore a commutative $\I$-rig 
%then $\uhocolim_{\I} X$ is an
%$E_\infty$-ring space.
\end{theorem}

Without further assumptions on $X$, this $E_{\infty}$-space need not be group-like 
(i.e. the monoid $\pi_0(\uhocolim_{\I} X)$ need not be a group).  However, we can always 
form the group completion $\Omega B (\uhocolim_{\I} X)$ to get the associated 
infinite loop space.
Note that an algebra over the Barratt-Eccles operad has an underlying monoid structure that is always strictly associative (and homotopy commutative) so that the usual functorial
construction of the classifying space for monoids built using the bar construction can be applied.  We will always use this model for $B$ in defining the group completion functor $\Omega B(-)$.  The consistency results in \cite{May1} guarantee that the group 
completion $\Omega B (\uhocolim_{\I} X)$ defines an infinite loop space weakly equivalent 
to that obtained using any other delooping machine.

In \cite{CS}, Schlichtkrull defined a different infinite loop space associated to $X$, 
using the language of $\Gamma$-spaces.  
Schlichtkrull's construction is the same 
as May's construction \cite{May2} 
of a $\Gamma$-space applied to the permutative 
category $\I \ltimes X$.  By the uniqueness result of \cite{May2}, the infinite 
loop space $\Omega B (\uhocolim_{\I} X)$ is equivalent to that defined by Schlichtkrull.

%When $X$ is a commutative $\I$-rig, we process the associated bipermutative category $\I \ltimes X$ using the machinery of Elmendorf-Mandell.  To a bipermutative category $\mathcal{C}$, they functorially associate a commutative symmetric ring spectrum $K\mathcal{C}$ \cite[Cor 3.9, Thm 9.3.8]{EM} for which $\Omega^{\infty} K \mathcal{C} \simeq \Omega B B\mathcal{C}$.  By \cite[Thm 1.2, 1.3]{Lind}, this infinite loop space is equivalent to the $E_{\infty}$-ring space underlying the associated $E_{\infty}$-ring spectrum, in the sense of \cite[VI]{MQR}. We can now summarize with the following theorem.

\bigskip

We now give a different construction of an infinite loop space associated 
to $X$. 
To start  note the decomposition of categories
$$
\P \ltimes X =\bigsqcup_{n\ge 0}\Sigma_{n}\ltimes X(\n),
$$ 
where $\Sigma_{n}$ is seen as a category with one object.   
Thus $\P \ltimes X$ is a topological category with classifying space 
\[
M:=\uhocolim_\P X =
B(\P\ltimes X)\simeq\bigsqcup_{n\ge 0}E\Sigma_{n}\times_{\Sigma_{n}}X(\n).
\]
As $\P\ltimes X$ is a permutative category, $M=B(\P \ltimes X)$ 
is an $E_{\infty}$-space and thus its group completion, $\Omega BM$, is an 
infinite loop space.  
The reduction maps $X(\n ) \to *$  define a map of permutative categories 
$ \P \ltimes X \to \P \ltimes *$ and hence
a map of infinite loop spaces
\[
\rho ^X: \Omega B (\uhocolim_\P  X)  \longrightarrow \Omega B (\uhocolim_\P
 *). 
\]
In particular, the homotopy fiber $\hofib \rho^X$ is naturally an infinite loop space.

When $X$ is a commutative $\I$-rig, we process the associated bipermutative category $\P \ltimes X$ using the machinery of Elmendorf and Mandell.  To a bipermutative category $\mathcal{C}$, they functorially associate a commutative symmetric ring spectrum  \cite[Cor 3.9, Thm 9.3.8]{EM}.  By \cite[Thm 4.6]{EM} and the original work of Segal \cite{Segal}, its underlying infinite loop space is weak homotopy equivalent to $\Omega B B\mathcal{C}$.  By a theorem due to Schwede \cite{Schwede}, and later refined by Mandell and May \cite[Sec.1]{MM}, the homotopy category of commutative symmetric ring spectra is equivalent to that of $E_{\infty}$ ring spectra.  We write $K\mathcal{C}$ for the $E_{\infty}$ ring spectrum associated to $\mathcal{C}$ under this equivalence of homotopy categories.  The underlying infinite loop space of an $E_{\infty}$ ring spectrum is an $E_{\infty}$ ring space, as defined in \cite[VI]{MQR}, so we may functorially associate to each bipermutative category an $E_{\infty}$ ring space $\Omega^{\infty}K\mathcal{C}$.  Moreoever, by \cite[Thm 1.2]{Lind}, the space $\Omega^{\infty}K\mathcal{C}$ is weak homotopy equivalent to the group completion $\Omega B B \mathcal{C}$.

%By a mild abuse of terminology, we summarize the situation by saying that $\Omega B B\mathcal{C}$ is an $E_{\infty}$ ring space.  The symbol $\Omega^{\infty}$ denotes a homotopy invariant version of the underlying infinite loop space functor, as defined in \cite{Lind}, for example.  By \cite[Thm 1.2, 1.3]{Lind}, the space $\Omega B B\mathcal{C}$ has the same weak homotopy type as the $E_{\infty}$-ring space underlying the $E_{\infty}$-ring spectrum associated to $K\mathcal{C}$, in the original sense of \cite[VI]{MQR}.  This justifies our use of the terminology.

We now apply this machinery to the morphism $\P\ltimes X \to \P \ltimes *$ of bipermutative categories.  We obtain a map of $E_{\infty}$ ring spectra 
\[
K(\P \ltimes X ) \longrightarrow K(\P \ltimes \ast)
\]
which is equivalent to $\rho^{X}$ after applying $\Omega^{\infty}$.  The homotopy fiber of a map of $E_{\infty}$ ring spectra is a non-unital $E_{\infty}$ ring spectrum.  By a non-unital $E_{\infty}$ ring space, we mean the underlying infinite loop space of a non-unital $E_{\infty}$ ring spectrum.  Since $\Omega^{\infty}$ preserves homotopy fiber sequences, this means that the homotopy fiber of a map of $E_{\infty}$ ring spaces is a non-unital $E_{\infty}$ ring space.  We have proved the next theorem.

\begin{theorem}\label{thm:hofib_ringspace}
For any commutative $\I$-monoid $X$ the homotopy fiber $\hofib \rho^X$ of 
\[
\rho ^X: \Omega B (\uhocolim_\P  X)  \longrightarrow \Omega B (\uhocolim_\P
 *). 
\]
is an infinite loop space.
If furthermore $X$ is a commutative $\I$-rig, then $\hofib \rho^X$ 
is a \john{non-unital} $E_\infty$-ring space.
\end{theorem}

\subsection{The main example}\label{subsec:examples}
For any group $G$,
conjugation by $G$ or action by any other  automorphism of $G$ 
induces a well-defined action on 
$B_n(q, G)= \Hom (F_n/\Gamma_n^q, G)$ by post-composition.
The action is also compatible with the simplicial face and degeneracy maps
in the bar construction and hence induces an action on $B(q, G)$.

For every $q\ge 2$ we define an $\I$-space
$B(q,U(-))$  by setting
$\n\mapsto B(q,U(n))$ with morphisms induced by the natural
inclusions and the action of $\Sigma_{n}$ on $B(q,U(n))$ 
given by conjugation through permutation matrices. 
Being induced by the natural action of $\Sigma _n$ on $\n$, 
it can be checked that this compatible sequence defines indeed an $\I$-space.

We give $B(q, U(-))$ the structure of an $\I$-monoid by defining 
the unit map $\eta_\n \colon * \rightarrow B(q, U(n))$ to be the inclusion 
of the base-point and defining the monoid structure map
\[
\mu_{\n,\m}:B(q,U(n))\times B(q,U(m))\to 
B(q,U(n+m))
\]
to be induced by the block sum of matrices.
To see that $\mu_{\n, \m}$ is well-defined note that block sum defines 
a group homomorphism $U(n) \times U(m) \to U(n+m)$. 
When taking elements of the symmetric groups to permutation matrices, 
the disjoint union of sets corresponds to 
block sum of matrices. Thus  $\mu$ defines a natural transformation 
of functors defined on $\I\times \I$.  One checks compatibility with
$\tau$ and hence $B(q, U(-))$ is a 
commutative $\I$-monoid.

Next we note that tensor product of matrices induces a well-defined map
\[
\pi_{\n, \m}: B(q, U(n)) \times B(q, U(m)) \longrightarrow B(q, U(nm)).
\]
To see this note that tensor product commutes with matrix multiplication and
hence induces a homomorphism $U(n) \times U(m) \to U(nm)$.
The map is equivariant for the symmetric group actions 
because the permutation matrix associated to  the product of two permutations
 is the same as the tensor product of the corresponding 
permutation matrices. \john{Hence $\pi$ is a natural transformation of functors defined on the category $\P \times \P$.  Note, however, that $\pi$ is not natural for \ul{proper} injections.  The map $\pi$ is compatible with $\tau$ and the
distributivity of block sum and tensor product of matrices 
induces distributivity maps for $\mu$ and $\pi$.} We have shown

\begin{theorem}
$ B(q, U(-))$  is a commutative $\I$-rig.
\end{theorem}

As a consequence, we may apply Theorem \ref{loop hocolim} and 
Theorem \ref{thm:hofib_ringspace} to get a pair of \john{infinite loop spaces, the latter of which carries a non-unital $E_{\infty}$-ring structure.
In the next section, we will show that  these two infinite loop spaces are equivalent.}

\section{Identifying and comparing the infinite loop spaces}\label{identifying}
Let $X$ be a commutative $\I$-monoid. We 
will first identify $\hofib \rho^X$ under  certain assumptions and then show 
it is homotopy equivalent as an infinite loop space to
$\uhocolim_\I X$.

Consider the space
\[
X_{\infty} := \uhocolim_{n \in \N}X(\n).
\] 
Note that $X_\infty \simeq \ucolim _{n\in \N} X(\n)$ if the structural 
maps $j_n: X(\n) \to X(\n + {\bf 1})$ are cofibrations.
In our applications this will always be the case.
Let  $X_{\infty}^{+}$ denote
Quillen's plus construction applied with respect to the
maximal perfect  subgroup of $\pi_1 (X_{\infty})$ (which we take to be 
understood to be done in each component separately, 
if $X_\infty$ is not connected).
Also recall that a space $Z$ is abelian if $\pi_1 (Z)$ is abelian and acts
trivially on homotopy groups $\pi_* (Z)$.
It is well-known that $H$-spaces are abelian. 

\begin{theorem}\label{loop permutative}
Let $X:\I\to \Top$ be a commutative $\I$-monoid. Assume 
that 
\begin{itemize}
\item  the action of $\Sigma_{\infty}$ on $H_*(X_{\infty})$ is  
trivial;
\item  the inclusions induce natural isomorphisms $\pi_0 (X(\n )) \simeq \pi _0(X_\infty)$
of finitely generated abelian groups with multiplication
compatible with the  Pontrjagin product and in the centre of the homology Pontrjagin ring;
\item 
the commutator subgroup of $\pi_{1}(X_{\infty})$ is perfect (for each component)
 and 
$X_{\infty}^{+}$ is abelian.
\end{itemize}
Then 
$\hofib \rho^X \simeq X_{\infty}^{+}$, and in particular $X_\infty ^+$ is an infinite loop space.
\end{theorem} 

\begin{proof}
Let $ M = \uhocolim _\P X= B(\P\ltimes X)$ and $m$ 
be the point corresponding to the base 
point in $X({\1})$ (in the identity component of $\pi_0 (X(\1))$). 
%Furthermore, let $m_1, m_2, \dots$ be generators of $\pi_0 (X(\n))$ and put
%$m = m_0 m_1 m_2 \dots m_k$. 
Then 
\[
\text{Tel}(M\stackrel{\cdot m}
{\rightarrow} M\stackrel{\cdot m}
{\rightarrow} M\stackrel{\cdot m}
{\rightarrow} \cdots)\simeq \Z\times (E\Sigma_{\infty}\times_{\Sigma_{\infty}}X_{\infty}).
\]
As $\P \ltimes X$ is a symmetric monoidal category, its classifying space $M$ is a 
homotopy commutative topological monoid.
The hypotheses imply that 
$
\pi_0 (M) 
$ 
is in the centre of $H_*(M)$. Hence $H_*(M)[\pi_0(M)^{-1}]$ can be constructed by 
right fractions, so that we may apply the group completion theorem \cite{MS, RW}.  Therefore there is a map 
\[
f:\Z\times (E\Sigma_{\infty}\times_{\Sigma_{\infty}}X_{\infty})\to \Omega BM
\]
which induces an isomorphism on homology with all systems of local coefficients on $\Omega BM$.  
Furthermore, the fundamental group 
(of each component) of $E\Sigma_{\infty} \times_{\Sigma_{\infty}} X_{\infty}$ 
has a perfect commutator subgroup by \cite{RW}, 
and $f$ extends to a homology equivalence
between abelian spaces  
\[
{f}^+:\Z\times  (E\Sigma_{\infty}\times_{\Sigma_{\infty}}X_{\infty})^{+}\to 
\Omega BM,
\]
which is thus a homotopy equivalence.
This shows in particular that $\Z\times  (E\Sigma_{\infty}\times_{\Sigma_{\infty}}X_{\infty})^{+}$
is an infinite loop space as $\Omega BM$ is  the group completion of an
$E_\infty $-space.

Consider now the 
fibration sequence 
\begin{equation}\label{ses1}
X_{\infty}\to E\Sigma_{\infty}\times_{\Sigma_{\infty}}X_{\infty}\stackrel{p}{\to} B\Sigma_{\infty}.
\end{equation}
and the associated map of plus constructions
\[
p^{+}:\Z\times  (E\Sigma_{\infty}\times_{\Sigma_{\infty}}X_{\infty})^{+}\to 
\Z\times B\Sigma_{\infty}^{+}.
\]
Since $f^+$ is a  homotopy equivalence and $\Omega B(\uhocolim_\P *) \simeq
\Z \times B\Sigma_\infty^+$, we can identify the  homotopy fiber of
$p^+$ with  $\hofib \rho^ X$.
By assumption the action of  $\Sigma_{\infty}$ on $X_{\infty}$ is homologically trivial.
We are also assuming that $X_{\infty}^{+}$ is abelian and in particular nilpotent.
Under these conditions the fibre sequence (\ref{ses1}) remains 
a fibre sequence after passing to  plus constructions, see \cite[Thm. 1.1]{Berrick}. 
Thus we have a homotopy fibration
\[
X_{\infty}^{+}\to \Z \times (E\Sigma_{\infty}\times_{\Sigma_{\infty}}X_{\infty})^{+}
\to \Z\times B\Sigma_{\infty}^{+}.
\]
This shows that the homotopy fibre of $p^{+}$ is  $X_{\infty}^{+}$ and so $X^+_\infty \simeq
\hofib \rho^X$.
\end{proof}

\begin{remark}
For any commutative $\I$-monoid
$X$ the multiplication on $M_X:= \bigsqcup _{n\geq 0} X(\n)$
is commutative up to the action of the shuffle maps $\tau_{\m, \n}$.
These are induced by the action of the symmetric group. So, assuming that 
these actions 
are trivial in homology, 
it follows that the Pontrjagin product is commutative on the 
level of homology. In particular $\pi_0 (M_X)$ is in the centre of the 
Pontrjagin ring $H_* (M_X)$. Thus by the group completion theorem \cite {MS}, 
the map 
$$
\mathbb Z \times X_\infty \longrightarrow \Omega B (M_X)
$$ 
is a homology
isomorphism.
In recent work \cite {Gritschacher}, Gritschacher has shown that without any further assumption, 
the commutator subgroup of $\pi _1 (X_\infty)$ is always perfect 
and that $X^+_\infty$ is always an abelian space. In other words, the 
assumptions in Theorem 3.1 on $\pi_1 (X_\infty)$ and $X^+_\infty$ are 
actually consequences\footnote{As we do not know whether $M_X$ is homotopy commutative, 
the results of
\cite{RW} cannot be applied directly to conclude that the induced map
$\mathbb Z \times X^+_\infty \to \Omega B (M_X)$  is a homotopy
equivalence.}. 

In contrast, the condition that the symmetric groups act 
homologically trivial is necessary. To see this consider 
the commutative $\I$-space $X$ with $X(\n):= Z^n$ for some pointed connected space $Z$.  
Then, by the parametrized version of the 
Barratt-Priddy-Quillen theorem (see for example \cite{May1}, \cite {Segal}), 
$$\Omega B( \uhocolim _\P X) \simeq Q(Z_+)$$ 
and thus $\hofib \rho^X 
\simeq \hofib p^+ \simeq Q( Z)$ while $X_\infty \simeq \uhocolim _n Z^n$.
Here $Q = \Omega ^\infty \Sigma ^\infty$ and $Z_+$ denotes the space $Z$ with an additional base point.  
\end{remark}

We now turn to the question of comparing the infinite loop spaces $\hofib \rho^X $ and $\uhocolim_{\I} X$.
Suppose that $X$ is a commutative $\I$-monoid.
Consider the 
following commutative diagram of strict functors between 
permutative categories
\[
\xymatrix{
\P \ltimes X\ar[r]^-{\alpha_{X}} \ar[d]_{\rho^X} 
& \I\ltimes X \ar[d]^{\rho^X_1} \\
\P \ltimes *\ar[r]^-{\alpha_{*}} &\I\ltimes *. 
}
\]
The horizontal maps are induced by the inclusion $\P \to \I$.
In the above diagram $*$ is  the terminal commutative $\I$-monoid 
and the vertical maps $\rho^X$ and $\rho^X_1$  are induced by  
the projection maps to a point. 
Passing to the level of classifying spaces and applying group completion 
we obtain a commutative diagram of infinite loop spaces 
\begin{equation}
\xymatrix{
\Omega B(\uhocolim _\P  X) \ar[r]^-{ \alpha_{X}} \ar[d]_{\rho^X} 
& \Omega B(\uhocolim_{\I} X) \ar[d]^{\rho^X_1} \\
\Omega B (\uhocolim_\P *)\ar[r]^-{\alpha_{*}} & \Omega B (\uhocolim_{\I} *)\simeq *. 
}
\end{equation}
Note that the empty set  is an initial object for  $\I$ and hence $\uhocolim_\I * = B\I \simeq *$.

The above 
diagram induces an infinite loop map between the homotopy fibers of the maps 
$\rho^X$ and $\rho^X_1$. 
By definition the homotopy fiber on the left  is the space $\hofib \rho^ X$.
Also, since $\uhocolim_{\I} *$ is contractible, the homotopy fiber 
on the right can be identified with $\Omega B (\uhocolim_{\I}X)$. 
This shows that we 
have a 
 map of infinite loop spaces
\[
\hofib \rho^X \overset g\longrightarrow \Omega B (\uhocolim_{\I} X).
\]
%If $X$ is a commutative $\I$-rig then $g$ is a map of $E_\infty$-ring spaces.
Note that $\rho^X$ has a canonical splitting of permutative categories induced by the unit $* \to X$ of the $\mathbb I$-monoid $X$.
Thus it follows from the following theorem that $g$ is a homotopy equivalence 
whenever the stated conditions on $X$ are satisfied.

\begin{theorem}\label{thm:compare_g}
Let $X$ be a commutative $\mathbb I$-monoid 
such that 
all maps $j: X(\n) \to X(\m)$ induced by injections $j: \n \to \m$
are  monomorphisms. 
Furthermore, assume that for all $x \in X(\n)$ and $y\in X(\m)$,
the sum $\mu_{\n, \m} (x,y) $ is in the image 
of a map induced by a non-identity order preserving injection if and only if $x$ or $y$ is. 
Then 
$$
\alpha_X \times \rho^X: 
\Omega B (\uhocolim _{\mathbb P} X ) 
\longrightarrow 
\Omega B( \uhocolim_{\mathbb I} X) \times \Omega B (\uhocolim _{\mathbb P} *)
$$
is a weak homotopy equivalence of infinite loop spaces 
which is natural for commutative $\mathbb I$-monoids.
%If $X$ is a commutative $\mathbb I$-rig, 
%then $\alpha_X \times \rho^X$ is a weak homotopy equivalence of 
%$E_\infty$-ring spaces which is natural for $\mathbb I$-rigs.
\end{theorem}

\john{Notice that when $X$ is a commutative $\I$-rig, we may use the theorem to transfer the non-unital $E_{\infty}$-ring space structure on $\hofib \rho^{X}$ along $g$ to obtain a non-unital $E_{\infty}$-ring space structure on the group completion of $\uhocolim_{\I} X$.}

A version of the theorem was proved 
by Fiedorowicz and Ogle \cite  {FiedoroOgle} in the 
setting of simplicial sets. 
This was revisited in Gritschacher \cite [Section 4]  {Gritschacher}. 
For  convenience of the reader we sketch a streamlined argument following
\cite {Gritschacher}.

\begin{proof}
Given $x\in X(\n)$ we can write it as $x= j_x(\bar x)$ where
$\bar x \in X(\bar \n)$, $j_x: \bar \n  \to \n $ is an order preserving 
injection,  
and $\bar n$ is minimal. 
We call $ x$ reduced if $x=\bar x$. Note that 
$\bar x$ and $j_x$ are uniquely determined. Denote by $\bar X(\n)$ 
the set of reduced elements in $X(\n)$. The assignment $\n \mapsto \bar X (\n)$
defines a $\mathbb P$-diagram. 
By the assumption on $\mu$ the commutative 
$\mathbb I$-monoid structure of $X$ induces the structure of a permutative category 
on $\mathbb P \ltimes \bar X$. 

Assume now that  $X$ is discrete. Then the assignment
$(\n, x) \mapsto (\bar \n , \bar x)$ on objects extends to define 
a 
functor 
$$R_X: \mathbb I\ltimes X \longrightarrow \mathbb P\ltimes \bar X.
$$ 
It has a right 
inverse given by the inclusion $\iota_X: \mathbb P \ltimes \bar X \to
\mathbb I \ltimes X$. 
Furthermore, the maps $j_x$ define a natural transformation from
$\iota_X \circ R_X$ to the identity on $\mathbb I \ltimes X$. 
Hence, $R_X$ defines a homotopy deformation
retract 
on classifying spaces. We also note that by our assumption on $\mu$, the functor $R_{X}$ is a strict symmetric monoidal functor.

The inclusions $\mathbb P \ltimes \bar X \to \mathbb P \ltimes X$ and $\mathbb P
\to \mathbb P \ltimes X$ combine via the monoidal product  functor 
to a functor 
$$
T_X: (\mathbb P \ltimes \bar X) \times \mathbb P \longrightarrow \mathbb P \ltimes X
$$
that maps the object $((\bar \n, \bar x), \n)$ to 
$(\bar \n + \n, j(\bar x))$, where $j$ is the 
canonical inclusion $\bar \n 
\hookrightarrow \bar \n + \n$. We claim this is a homotopy equivalence on classifying spaces. Indeed, an analysis of the effect of permutations on reduced points shows that the functor is bijective on automorphism groups of objects.  As both 
source and target categories are groupoids and every isomorphism class 
of the target category 
has a representative in the image, this is an equivalence of categories.
We note that $T_X$ is not a strict monoidal functor (only up to conjugation by a block permutation). However, the left inverse functor $(\n, x) \mapsto ((\bar \n, \bar x), 
\n - \bar\n)$ does commute strictly with the monoidal structure.
Hence, this defines a homotopy equivalence of monoids on classifying spaces, and induces a homotopy equivalence of  group completions.
Compare \cite [Lemma 1.7]{FiedoroOgle}.

Consider now the map of permutative categories
$$
\alpha_X \times \rho ^X: \mathbb P \ltimes X \longrightarrow (\mathbb I \ltimes X) \times \mathbb P
$$
and take the group completion of their classifying spaces
\begin{equation}\label{eq:claim_map}
\alpha_X \times \rho ^X: \Omega B (B(\mathbb P \ltimes X)) \longrightarrow 
\Omega B (B(\mathbb I \ltimes X)) \times \Omega B (B \mathbb P).
\end{equation}
We claim that this 
is a weak homotopy equivalence which is natural in
commutative $\mathbb I$-monoids. To see this precompose with the map of group completed classifying spaces induced by $T_{X}$ and postcompose with the map induced by $R_{X} \times \mathrm{Id}$.  The resulting composite is homotopic to the endofunctor of $(\mathbb P \ltimes \bar X) \times \mathbb P$ given by 
\[
((\bar \n, \bar x), \m) \longmapsto ((\bar \n, \bar x), \bar \n + \m).
\]
This map is the identity on the first component and an equivalence on the second component because we are working with group-complete monoids.  

%To see this  consider the commutative diagram (2).
%As $T_X$ is a homotopy equivalence also on group completions, the homotopy fibre of the left vertical arrow $\rho ^X$
%can be identified as
%$$
%\hofib \rho^X \simeq \Omega B (\mathbb P \ltimes \bar X)
%$$
%and the map $g$ of homotopy fibres is induced by the natural inclusion of monoidal categories $\iota_X : \mathbb P \ltimes \bar X \to \mathbb I \ltimes X$, which is a homotopy equivalence as we have argued before.

Using the naturality of the weak homotopy equivalence in \eqref{eq:claim_map} and applying it to 
boundary and face maps allows us to extend it to
$\mathbb I$-diagrams in 
simplicial sets. More precisely, for any commutative $\mathbb I$-monoid
$X$ 
in simplicial
sets, that satisfies level wise the condition on $\mu$, 
we have a map of simplicial permutative categories
which is a weak homotopy equivalence on applying $\Omega B (B(-))$ to each simplicial level, and hence a weak homotopy equivalence
on total  
spaces: 
$$
\alpha_X \times \rho^X: 
| \n \mapsto \Omega B (B  (\mathbb P \ltimes  X(\n)) )|
\simeq |\n \mapsto \Omega B (B(\mathbb I \ltimes X(\n))) \times \Omega B (B \mathbb P )| .
$$
As $\Omega$ commutes with Cartesian product, and as $|\n \mapsto \Omega Z(\n)|
\simeq \Omega | \n \mapsto  Z(\n) |$ whenever each $Z(\n)$ is connected (cf. \cite[12.3]{May}),
we also have
$$
\alpha_X \times \rho^X: 
\Omega | \n \mapsto  B (B  (\mathbb P \ltimes  X(\n)) )|
\simeq \Omega |\n \mapsto  B (B(\mathbb I \ltimes X(\n))) \times B (B \mathbb P )| .
$$
Furthermore,  as realizations of multi-simplicial sets can be taken in any order, we deduce that
$$
\alpha_X \times \rho^X:\Omega B (B  (\mathbb P \ltimes |\n \mapsto  X(\n)|) )
\simeq \Omega B (B(\mathbb I \ltimes |\n \mapsto X(\n)|)) \times \Omega B (B \mathbb P ) .
$$
Compare \cite [Lemma 1.8] {FiedoroOgle}.
Finally, by replacing every space by its singular simplicial set,
any $\mathbb I$-diagram $X$ in topological spaces gives rise to 
an $\mathbb I$-diagram in simplicial sets, taking commutative $\mathbb I$-monoids to simplicial ones. 
Note that the conditions on $\mu$ 
are pointwise conditions and are automatically satisfied by the singular $p$-simplices for each $p$.
As a space is weakly homotopy equivalent
to the realisation of its singular simplicial 
set, the theorem follows.     
\end{proof}

\begin{example}
Consider 
the commutative $\I$-space $X$ with $X(\n):= Z^n$, where $Z$ is a well-pointed 
connected space. 
Note that in this case $ \Sigma _n$
does not act trivially on $H_* (Z^n)$ and hence 
Theorem 3.1 does not apply. 
As before, by the parametrized version of the Barratt-Priddy-Quillen theorem, 
$$
\Omega B (\uhocolim _\P X) \simeq Q(Z_+) \simeq Q(\S^0) \times Q(Z)
$$  
and hence $\hofib \rho^X \simeq Q(Z)$. 
Thus, by the above theorem,
we also have $\uhocolim _\I X \simeq Q(Z)$, which is in agreement with 
a result of Schlichtkrull \cite {CS2}.
\end{example}

\section{Constructing filtrations by infinite loop spaces}\label{applications}

In this section we use the results obtained in the previous sections to produce
filtrations of classical infinite loop spaces by sequences of infinite loop spaces
arising from the descending central series of the free groups.

\begin{theorem}\label{cor infinite loop spaces}
The spaces $B(q, U)$, $B(q,SU)$,  $B(q,SO)$ 
$B(q,O)$ 
and  $B(q, Sp)$ provide a 
filtration by \john{non-unital} $E_\infty$-ring spaces of the classical \john{non-unital} $E_\infty$-ring spaces 
$BU$, $BSU$, $BSO$,
$BO$ 
and $BSp$ respectively. 
\end{theorem}

\begin{proof}
Consider first the case of $BU$. Recall that the spaces $B(q,U)$ 
provide a filtration of the space $BU$ 
\[
B(2,U)\subset B(3,U)\subset \cdots\subset B(q,U)\subset B(q+1,U)\subset \cdots 
\subset BU.
\]
We will show that this filtration is a filtration by \john{non-unital} $E_\infty$-ring spaces. 
For this notice that by the main example in Section 2, each  $\n \mapsto B(q, U(n))$ for $q\geq 2$ 
is a commutative  $\I$-rig. In what follows we are going to show  that 
the conditions of Theorem \ref{loop permutative} are satisfied, and hence
$B(q,U)\simeq \hofib \rho^ {B(q, U(-))}$ is a \john{non-unital} $E_\infty$-ring space by Theorem \ref{thm:hofib_ringspace}.

The conjugation action of $\Sigma_{n}$ on $B(q,U(n))$ is homologically trivial 
because this action factors through the conjugation action of 
$U(n)$. The conjugaton action by any element in $U(n)$ is trivial, 
up to homotopy, since the action of the identity matrix is trivial and
$U(n)$ 
is path-connected. 
This implies in particular that the action of 
$\Sigma_{\infty}$ on $B(q,U)$ is  homologically trivial.   

Note that $B(q,U(n))$ and hence $B(q,U)$ is path connected. Next, we argue that the space $B(q,U)$ 
is an H-space under direct sum multiplication.
To be more precise, consider 
the injection $\mathbb N \sqcup \mathbb N \to \mathbb N$
defined by $(1,2,3,4,\dots ) \cup (1',2',3',4', \dots )
\mapsto (1,2,1',2',3,4,3',4', \dots)$.
It  defines a map of vector spaces 
$\mathbb C^\infty \times \mathbb C^\infty \to \mathbb C^\infty$ 
and hence a continuous homomorphisms $U \times U \to U$.
The image of $U(n)$ in $U$ under right or left multiplication by the identity 
matrix $I$ 
differs from the 
image under the standard inclusion by conjugation of an 
even permutation. As such a permutation is in the path-component of the identity matrix, we see that the multiplication is unital up to homotopy.

H-spaces have abelian fundamental group and hence
Theorem \ref{loop permutative} applies.   We conclude 
that $B(q,U)\simeq \hofib \rho^{B(q, U(-)) }$
for every $q\ge 2$ and is a \john{non-unital} $E_{\infty}$-ring space by 
Theorem \ref{thm:hofib_ringspace}.
The very same arguments can be used to prove 
analogous statements for the commutative $\I$-rig  $\n \mapsto B(q,SU(n))$, 
and $\n \mapsto B(q, Sp (n))$ 
for any $q\ge 2$.

In case of the commutative $\I$-rig 
$\n \mapsto B(q, SO(n))$ we note that $\Sigma_n$ is not a subgroup of $SO(n)$. 
Nevertheless, the alternating group $A_n$ is contained in $SO(n)$ and 
by the same argument as above acts therefore 
trivially on the homology of $B(q, SO(n))$.
Furthermore, when $n$ is odd, any odd permutation is represented by a matrix with determinant 
equal to $-1$. Hence it can be path-connected to the diagonal matrix $-I$ with constant entry $-1$. 
As $-I$ is in the center of $O(n)$ it acts trivially by conjugation on $B(q,SO(n))$ and 
hence also on its homology. 
But then so does any odd permutation. This proves that 
when $n$ is odd the action of $\Sigma_{n}$ on $B(q,SO(n))$ is homologically trivial. 
This in turn implies that the action of $\Sigma_{\infty}$ on $B(q,SO)$ is homologically trivial. 
We also have that $B(q,SO)$ is an H-space and hence abelian. Thus 
$B(q,SO)\simeq \hofib \rho^{B(q, SO(-)) }$
for every $q\ge 2$ and it is a \john{non-unital} $E_{\infty}$-ring space by 
Theorem \ref{thm:hofib_ringspace}.
This line of argument can also be used to prove the analogous statement for 
the commutative $\mathbb I$-rig  $\n \mapsto B(q, O(n))$.

\end{proof}

%\begin{remark}
%There is an alternative proof of the theorem that uses the theory of orthogonal 
%spaces developed in \cite{Lind}.  Each of the $\I$-spaces $B(q, G(-))$ used in 
%the proof of the theorem has a canonical extension to a functor $\mathcal{I} \to \Top$ 
%indexed on the category $\mathcal{I}$ of finite dimensional inner product spaces 
%and linear isometries.  It then follows from \cite[Prop 9.4]{Lind} and a version of Quillen's 
%Theorem A for topological categories (for example, \cite[A.5]{Lind}) that the canonical map %$\uhocolim_{\N} B(q, G(-)) \to \uhocolim_{\I} B(q, G(-))$ 
%is a weak homotopy equivalence.  This is not true for every $\I$-space, 
%but it is true for those $\I$-spaces that come from orthogonal spaces.  
%We then conclude from Theorem \ref{loop hocolim} that 
%$B(q, G) \simeq \uhocolim_{\I} B(q, G(-))$ provides a filtration of $BG$ by infinite loop %spaces.

%However, when the group $G$ is discrete, as in the case of algebraic 
%$K$-theory and the sphere spectrum considered below, we do not have an 
%extension to orthogonal spaces, and this proof cannot be used.  
%It is for this reason that we have developed the material in 
%Section 3 and prefer the given proof using Theorem 3.1.
%\end{remark}

As remarked in \cite[Thm. 6.3]{ACT}, 
the natural map 
$\Omega B(q,G)\to \Omega BG$ admits a splitting up to homotopy. It is given by 
a factorization of the usual homotopy equivalence 
$G  \to \Omega BG$. Indeed we have that 
$\Sigma G = F_1B(q,G) = F_1BG$, where $F_1$ denotes the  first
layer in the usual filtration of the geometric realization of these simplicial
spaces. Hence, the  adjoint of $\Sigma G \to BG$ 
factors through $\Omega B(q, G)$. 
Note that this splitting does not in general admit a delooping; see 
\cite[Section 6]{ACT} for a counter-example.
Nevertheless, we have the following theorem. Here
$E(q,G)$ denotes the pull-back of the 
universal $G$-bundle $EG$ over $BG$.
It is homotopy equivalent to the homotopy fiber of the
inclusion $B(q,G) \to BG$.

\begin{theorem}\label{splitting_thm}
For all $q\ge 2$, and $G=U, SU, SO, O$ and $ Sp$, there is a homotopy split 
fibration of infinite loop spaces
$$
E(q, G) \longrightarrow B(q, G) \longrightarrow BG.
$$
In particular there is  a splitting of 
spaces 
\[
B(q,G) \simeq BG \times E(q,G).
\]
Both are natural in the entry $q$, 
meaning that both are compatible with  the filtration maps.
\end{theorem}

In order to prove the theorem, we will need to know the 
fundamental group of $B(q,G)$ for the groups in question. We have the following result general result.

\begin{lemma}\label{fundamental group}
Let $G$ be a topological group with a CW-structure.
Assume $\pi_0 (G)$ is abelian and that the natural homomorphism $G\to \pi_0(G)$ splits. 
Then for all $q\geq 2$
$$\pi _1 (B(q, G)) = \pi_0 (G).$$
\end{lemma}
\begin{proof}
Consider $\Sigma G = F_1B(q,G) = F_1BG$. As the 1-skeleton of the realisation of a (good) simplicial space is contained in the first filtration \cite [Prop. 11.4] {May}, any map from $S^1$ to $B(q,G)$ will factor through $\Sigma G$. Hence the map $\Sigma G \to B(q,G)$ is surjective on fundamental groups. 
 
The fundamental group of a suspension $\Sigma X$ for any space $X$ has fundamental group the free group over
the set $\pi_0(X) -\{1\}$; hence we have
$$
\pi_1 (\Sigma G) =  F(g | \, g\in \pi_0 (G)-\{1\}).
$$
The inclusion $\Sigma G \to BG$ induces the surjective map of fundamental groups 
$\pi_1 (\Sigma G) \to \pi_0 (G)$ which sends a generator $g$ to the \textsl{element}  $g\in \pi_0(G)$, and more generally the word $g_1 \bullet \dots \bullet g_k$  to the product of the elements $g_1 \dots g_k$.   To see this
geometrically, consider $\pi_0(G)$ as a subgroup of $G$, and note that the 2-simplex $(g,h)$ defines a homotopy from the two--letter word  $g\bullet h$ to the
product element $gh$. 

We now note that as $\pi_0 (G)$ is abelian, the 2-simplex
$(g, h)$ is contained in $B_2(q, G)$ for $q\geq2$. Hence all the above relations are already satisfied in $\pi_1 (B(q, G))$. As the factorization $\pi_1 (\Sigma G) \to \pi_1(B(q, G))\to \pi_1(BG)$ is surjective, the result follows.
\end{proof} 

\begin{proof}[Proof of Theorem \ref{splitting_thm}]
As $EG_{\infty} \simeq *$, for every $q\ge 2$ we have a 
homotopy fibration sequence 
$E(q,G_{\infty})\to  B(q,G_{\infty})\to  BG_{\infty}$.
As the map on the right is a map of infinite loop spaces, the homotopy 
fiber $E(q,G_{\infty})$ is an infinite loop space.
It remains to show that it splits.

Let $G_{n}$  
denote one of the groups $U(n)$, $SU(n)$,  $SO(n)$,  $O(n)$ or $Sp(n)$ so that 
$G_{\infty}=\ucolim_{n}G_{n}$ denotes  the group $U$, $SU$, $SO$, $O$ or $Sp$  respectively. 
For each fixed $q\ge 2$, the assignment $\n\mapsto \Omega B(q,G_{n})$ 
defines a commutative $\I$-rig with $\mu$ given 
by block sum 
and $\pi$ given by tensor product of matrices. 
In the same way the assignment 
$\n\mapsto \Omega BG_{n}$ also defines a commutative $\I$-rig 
and the inclusion map $\Omega B(q,G_{n}) \to \Omega BG_{n}$ defines 
a morphism of commutative $\I$-rigs. 

We claim that the 
commutative $\I$-rigs $G_{-}$, $\Omega B(q,G_{-})$ and $\Omega BG_{-}$ 
satisfy the hypotheses of Theorem \ref{loop permutative}. 
Indeed, except in the case $G = O$,
$G_{n} \simeq \Omega BG_{n}$ is path-connected for every $n \geq 0$ and,   
as $\pi_{0}(\Omega B(q,G_{n}))\cong\pi_{1}(B(q,G_{n}))$ 
is trivial by Lemma \ref{fundamental group}, 
$\Omega B(q,G_{n})$ is also path-connected. When $G=O$, 
$$
\pi_0 (\Omega B(q, O(n))) = \pi_1 B(q, O(n)) = \Z/2\Z
$$   
for each $n \geq 1$ by Lemma \ref{fundamental group}.
The multiplication in $\pi_0 \Omega B (q, O(n))$ 
is compatible with direct sum and stabilisation. 
This checks the second condition in Theorem \ref{loop permutative}.

Except in the cases $G =SO$ or $G=O$ the action of 
$\Sigma_{n}$ is homologically trivial as conjugation by any element in the path component of the identity is  trivial, up to homotopy, 
and  $G_{n}$ is path-connected. This implies that  
$\Sigma_{\infty}$ acts homologically trivial on $G_{\infty}$, 
$\Omega B(q,G_{\infty})$ and $\Omega BG_{\infty}$. The same conclusion 
can be obtained for  $G =SO$ or $G=O$ using a similar argument as 
in the proof of Theorem \ref{cor infinite loop spaces}. Hence the first condition from Theorem \ref{loop permutative} holds.

To verify the third condition, observe that the commutator group 
of $\pi_{1}(\Omega B(q,G_{n}))\cong \pi_{2}(B(q,G_{n}))$ is trivial, 
as this group is abelian in all cases. Finally,  
$\Omega B(q,G_{\infty})$ is an abelian space since it is a loop space and 
hence in particular  an $H$-space. 

By Theorem \ref{loop permutative}
we thus have maps of $E_{\infty}$-spaces
\[
G_{\infty}\to  \Omega B(q,G_{\infty})\to \Omega BG_{\infty}
\]
whose composition is a homotopy equivalence. 
Taking classifying spaces is compatible with  $E_\infty$-space structures 
and hence the above splitting deloops to give the  
splitting of the  theorem.
\end{proof}

We have concentrated so far on compact groups such as $O(n)$ and $U(n)$,
although the methods clearly extend to other linear groups.
Using some results by Pettet-Souto \cite{PS} and
Bergeron \cite{Bergeron} we can prove the following theorem.

\begin{theorem}\label{reduction to compact case}
Suppose that $G$ is the group of complex or real points in a reductive
linear algebraic group (defined over $\R$ in  the real case). Let
$K\subset G$ be a maximal compact subgroup. Then the inclusion map
$i:B(q,K)\to B(q,G)$ is a homotopy equivalence for every $q\ge 2$.
\end{theorem}

\begin{proof}
By \cite[Thm.I]{Bergeron}  it follows that the inclusion map
$i_{n}:B_{n}(q,K)\to B_{n}(q,G)$ is a homotopy equivalence for all
$q\ge 2$ and all $n\ge 0$. Thus the inclusion map
induces a simplicial map $i_{*}:B_{*}(q,K)\to B_{*}(q,G)$ that
is a level-wise homotopy equivalence. Since  
$G$ is assumed to be  the group of complex or real points in a reductive
linear algebraic group (defined over $\R$ in the the real case),  
we can identify $G$ with a Zariski
closed subgroup of $SL_{N}(\C)$ for some $N\ge 0$. Also, 
for every $n\ge 0$ we can see the space $B_{n}(q,G)$ as an algebraic 
variety since it is defined in terms of iterated commutators of elements in $G$ 
and such equations can be defined in terms of polynomial functions. Moreover, 
the subspace $S_{n}^{1}(q,G)\subset B_{n}(q,G)$ consisting of all 
$n$-tuples in $B_{n}(q,G)$ for which at least one of the coordinates 
is equal to $1_{G}$ is an algebraic subvariety of  $B_{n}(q,G)$. 
By the semi-algebraic triangulation theorem (see \cite[Section 1]{Hironaka}) 
it follows that   $B_{n}(q,G)$ has the structure of a CW-complex in  
such a way that $S_{n}^{1}(q,G)$ is a subcomplex. In particular, it follows that 
the pair $(B_{n}(q,G),S_{n}^{1}(q,G))$ is a strong NDR-pair. This proves 
that $B_{*}(q,G)$ is a proper simplicial space. The same is true for 
$B_{*}(q,K)$.  Using the
glueing lemma, for example see \cite[Thm. A.4]{May1},
we obtain the result of the  theorem.
\end{proof}

Our tools can also be used to obtain a similar filtration for the infinite 
loop space defining algebraic K-theory for any discrete ring $R$. 
Indeed, suppose that  $R$ is a discrete ring with unit and let 
$q\ge 2$. Consider 
the commutative   $\I$-rig  $B(q,GL_{-}(R))$ defined by
$\n\mapsto B(q,GL_{n}(R))$. As before the 
morphisms are induced by the natural 
inclusions and the conjugation action of $\Sigma_{n}$ on $B(q,GL_{n}(R))$.  
The multiplication map 
\[
\mu_{\n,\m}:B(q,GL_{n}(R))\times B(q,GL_{m}(R))\to 
B(q,GL_{n+m}(R))
\] 
is also given by the block sum  and $\pi$ by tensor product of matrices. 
Note that  Theorem \ref{thm:compare_g} applies to give
$$\uhocolim_{ \I}B(q,GL_{-}(R)) \simeq \hofib\rho^{ B(q, GL_{-} (R))}. $$
By Theorem \ref{thm:hofib_ringspace}, this space has the structure of a \john{non-unital}
$E_\infty$-ring space. 
This way we obtain a filtration of \john{non-unital}
$E_\infty$-ring spaces:
\[
\uhocolim_{\I}B(2,GL_{-}(R))\subset \cdots 
\subset \uhocolim_{\I}B(q,GL_{-}(R))\subset \cdots\subset 
\uhocolim_{\I}BGL_{-}(R).
\]
As is well known, the conjugation action of $\Sigma_{n}$ on $BGL_{n}(R)$ 
is homologically trivial. It follows from Theorems \ref{loop permutative} and \ref{thm:compare_g} that we have an 
equivalence 
$$
BGL_{\infty}(R)^{+} \simeq \hofib \rho^{ BGL_{-}(R)} \simeq \uhocolim_{\I}BGL_{-}(R).
$$ 
Thus the above 
gives a  filtration  of  \john{non-unital} $E_\infty$-ring spaces with final space weakly homotopy equivalent to the 
algebraic K-theory of $R$.
However,  unlike the case of $BGL_{n}(R)$, 
we do not know whether the conjugation action of 
$\Sigma_{n}$ on $B(q,GL_{n}(R))$   is  
homologically trivial, and we expect that  the natural 
map 
\[
B(q,GL_{\infty}(R))\to \uhocolim_{\I} B(q, GL_{-}(R))
\]
is not a homology isomorphism. 

In a similar way we can obtain a filtration of $Q(\S^{0})$. 
For this note that 
the conjugation action of 
$\Sigma_{n}$ on $B\Sigma_{n}$ is homologically trivial. 
Therefore, by the Barratt-Priddy-Quillen theorem, the level zero component of $Q(\S^{0})$ is equivalent to the 
homotopy colimit over $\I$ of the classifying spaces of the symmetric groups:
\[
Q_{0}(\S^{0})\simeq (B\Sigma_{\infty})^{+}\simeq \hofib \rho^{ B\Sigma _{-}}\simeq  \uhocolim_{\I}B\Sigma_{-}.
\]
Consider the commutative $\I$-rig 
$B(q,\Sigma_{-})$ defined by
$\n\mapsto B(q,\Sigma_{n})$. The structural maps are given by conjugation 
of $\Sigma_{n}$ and inclusions in an analogous way as above. 
Then by Theorem \ref{loop hocolim} 
we have  a filtration  of  \john{non-unital} $E_\infty$-ring spaces 
\[
\uhocolim_{\I}B(2,\Sigma_{-})\subset \cdots 
\subset \uhocolim_{\I}B(q,\Sigma_{-})\subset\cdots\subset \uhocolim_{\I}B\Sigma_{-}
\simeq Q_0(\S^{0}).
\]
As in the case of $B(q,GL_{n}(R))$, the conjugation action of $\Sigma_{n}$ 
on $B(q,\Sigma_{n})$ may fail to be homologically trivial
(for example this is the case for the conjugation action of $\Sigma_{3}$ on $B(2,\Sigma_{3})$,
see \cite{ACT}).
The conditions of Theorem \ref{thm:compare_g} are satisfied but the 
homotopy types of the spaces $\uhocolim_{\I}B(q,\Sigma_{-}) \simeq 
\hofib \rho^{B(q, \Sigma _{-})}$ remain to be determined.

\begin{corollary}
The spaces 
$$
\uhocolim_{\I}B(q,GL_{-}(R)) \simeq \hofib \rho^{ B(q, GL_{-}(R))}
\quad \text{ and } \quad 
\uhocolim_{\I}B(q,\Sigma_{-}) \simeq \hofib \rho^{B(q, \Sigma_{-})}
$$ 
provide filtrations of  \john{non-unital}
$E_\infty$-ring spaces with final target 
 the classical \john{non-unital} $E_\infty$-ring spaces 
$BGL_{\infty}(R)^{+}$ and $Q_0(\S^{0})$.
\end{corollary}

\section{Transitional nilpotence, bundles and K-theory}\label{K theory}

In this section we extend the notions of transitionally commutative bundles
and commutative K-theory as defined in \cite{AG} to more general 
$q$-nilpotent notions for $q\ge 2$, reflecting the filtration induced by the
descending central series of the free groups. We will show that these geometrically
defined theories are represented by the infinite loop spaces $\mathbb Z\times B(q,U)$.

\begin{definition}
For a CW-complex $X$ a 
principal $G$-bundle $\pi:E\to X$ is said to have \textsl{transitional nilpotency class} 
at most $q$ if there exists an open cover $\{U_{i}\}_{i\in I}$ of $X$ 
such that the bundle $\pi:E\to X$ is trivial over each $U_{i}$ and for 
every $x\in X$ the group generated by the collection $\{\rho_{i,j}(x)\}_{i,j}$ is a group of 
nilpotency class at most $q$. Here $\rho_{i,j}:U_{i}\cap U_{j}\to G$  denotes 
the transition functions and $i,j$ run through all indices in $I$ for which 
$x\in U_{i}\cap U_{j}$. The minimum of all such numbers $q$ is said to be 
transitional nilpotency class of  $\pi:E\to X$.
 \end{definition}

The principal $G$-bundle $p_{q}:E(q,G)\to B(q,G)$  
is  universal for all principal $G$-bundles with  
transitional nilpotency class less than $q$. 

\begin{theorem}\label{geometric interpretation} 
Assume that $G$ is an algebraic subgroup of $GL_{N}(\C)$ 
for some $N\ge 0$, $X$ is a finite CW-complex and that 
$\pi:E\to X$ is a principal $G$-bundle over $X$. Then, 
for any $q\ge 2$, the classifying map $f:X\to BG$ of $\pi$ 
factors through $B(q,G)$ (up to homotopy) if and only if 
$\pi$ has transitional nilpotency class less than $q$. 
\end{theorem}
\begin{proof}
The case $q=2$ was treated in \cite[Thm. 2.2] {AG} and 
in fact this theorem is true for any Lie group in this case. The 
proof goes through verbatim also for $q>2$ using the fact 
that when $G$ is an algebraic subgroup of $GL_{N}(\C)$,  
then the simplicial space $B_{*}(q,G)$ is proper as 
was pointed out in the proof of Theorem \ref{reduction to compact case}.
\end{proof}

As $[\Sigma X, BG] = [X, \Omega BG]$ and the canonical map $\Omega B(q,G) \to \Omega BG$
always admits a splitting up to homotopy, any principal $G$-bundle on a suspension $\Sigma X$
has transitional nilpotency class less than $q$ for all $q$. 
However, the nilpotency structure is not unique in general, not even up to 
isomorphism  in the sense of the following definition.

\begin{definition}
Let $\pi_{0}:E_{0}\to X$ and $\pi_{1}:E_{1}\to X$ be 
two principal $G$-bundles 
with transitional nilpotency class less than $q$. We say that 
these bundles are $q$-transitionally isomorphic if there exists a  
principal $G$-bundle $p:E\to X\times [0,1]$ with transitional nilpotency class 
less than $q$ such that 
$\pi_{0}=p_{|p^{-1}(X\times\{0\})}$ and $\pi_{1}=p_{|p^{-1}(X\times\{1\})}$.
\end{definition}

A complex vector bundle $\pi:E\to X$ is said to have transitional nilpotency class 
less than $q$ if the corresponding frame bundle, under a fixed Hermitian 
metric on $E$, has  transitional nilpotency class less than $q$. 
Theorem \ref{splitting_thm} can then be interpreted to say that any vector bundle is stably of
transitional nilpotency class less than $q$ for all $q\geq 2$,
and there is a functorial choice of such a structure.
The set $Vect_{q-nil}(X)$ of $q$-transitionally isomorphism 
classes of  complex vector bundles over $X$ with transitional nilpotency 
class less than  $q$ is a monoid under the direct sum of vector bundles.
The $q$-nilpotent K-theory of $X$ is defined as the associated Grothendieck group.

\begin{definition} 
$K_{q-nil}(X):=Gr(Vect_{q-nil}(X))$.
\end{definition}

Tensor products induce a natural multiplication on $K_{q-nil} (X)$ just as in classical $K$-theory.

\begin{theorem}
For any finite CW-complex $X$ there is a natural isomorphism of rings
\[ 
K_{q-nil}(X)\cong [X,\Z\times B(q,U)]. 
\]
Hence, it is the zeroth term of a multiplicative generalized cohomology theory.
\end{theorem}

\begin{proof}
Let $X$ be a finite CW-complex. By working one path-connected component at a time, 
we may assume without loss of generality that $X$ is path-connected.
By Theorem \ref{geometric interpretation},
$$
Vect _{q-nil} (X)  = [X, \bigsqcup_{n\ge 0}B(q,U(n))		]
$$ 
as abelian monoids where the addition is  induced by direct sum of matrices on the right hand side.
Any injection $\N \times \N  \to \N$ induces a linear injection $\C^\infty \times \C^\infty \to \C^\infty$
which in turn induces an  H-space product on $\Z \times B(q, U)$. The natural inclusions $B(q, U(n)) \to B(q, U)$ define a map
$$
[X, \bigsqcup_{n\ge 0}B(q,U(n))]
\longrightarrow  [X,\Z\times B(q,U)].
$$
As the symmetric groups act by homotopy equivalences  on $B(q,U)$, we see that the above map is compatible with the product structure 
on both sets, i.e. it is a map of monoids. By the universal property of the Grothendieck construction, this map factors through  
 a unique map of abelian groups
$$
K_{q-nil}(X) \longrightarrow [X,\Z\times B(q,U)].
$$
As $X$ is compact, any map $X \to B(q,U)$ factors through some $B(q, U(n))$ for some large enough $n$.  
Hence  the above map is surjective. 

To prove that it is injective, suppose that the image of $[A] - [B] \in K_{q-nil}(X)$ in $[X,\Z\times B(q,U)]$ is zero.  Let us write $f_{B} \colon X \longrightarrow B(q, U)$ for the image of a map representing $B$ in the colimit $B(q, U) = \ucolim_{n \in \mathbb{N}} B(q, U(n))$.  Since $B(q, U)$ is a group-like $H$-space, the induced product on $\mathrm{Map}(X, B(q, U))$ is also a grouplike $H$-space structure.  Let $f_{B'} \colon X \longrightarrow B(q, U)$ be a homotopy inverse for $f_{B}$ under this product.  Since $X$ is compact, we may factor $f_{B'}$ through a finite stage of the colimit and find a corresponding bundle $B'$ over $X$ with transitional nilpotency class less than $q$ which is classified by the map $f_{B'}$.  It follows that $B \oplus B'$ is stably $q$-transitionally isomorphic to the trivial bundle $\epsilon_{k}$ of rank $k = \dim B + \dim B'$.  By our assumption, we see that the image of $[A \oplus B'] - [\epsilon_{k}]$ in $[X,\Z \times B(q,U)]$ is also zero.  This means that $A \oplus B'$ is stably $q$-transitionally isomorphic to a trivial bundle, say $A \oplus B' \oplus \epsilon_{t} \cong \epsilon_{k + t}$.  We then have the relation
\[
[A] - [B] = [A \oplus B' \oplus \epsilon_t] - [\epsilon_{k + t}] = 0
\]
in $K_{q-nil}(X)$, which completes the proof.
\end{proof}

This answers the question raised in \cite{AG}
for $q=2$. 
Moreover, we have a sequence of cohomology theories and maps between them 
\[
K_{com}(X)=K_{2-nil}(X)\to K_{3-nil}(X)\to \cdots \to K_{q-nil}(X)\to \cdots \to K(X).
\]
By Theorem \ref{splitting_thm} topological K-theory splits off $q$-nilpotent K-theory for all $q\ge 2$.
These theories are not well understood and would seem to warrant further
attention. For example in \cite{AG} it was shown that 
$K_{com}(\mathbb S^i)\cong K(\mathbb S^i)$ for $0\le i\le 3$, but that
$K_{com}(\mathbb S^4)\ne K(\mathbb S^4)$.
 
We leave it to the reader to formulate $q$-nilpotent versions of
real and hermitian K-theory.

\end{document}